\documentclass[a4paper,12pt,reqno]{amsart}
\usepackage{amsfonts}
\usepackage{amsmath}
\usepackage{amssymb}
\usepackage[a4paper]{geometry}
\usepackage{mathrsfs}
\usepackage{xcolor}
\usepackage{hyperref}
\renewcommand\eqref[1]{(\ref{#1})} 
\numberwithin{equation}{section}
\theoremstyle{plain}
\newtheorem{thm}{Theorem}[section]
\newtheorem{prop}[thm]{Proposition}

\newtheorem{lem}[thm]{Lemma}

\theoremstyle{definition}

\newtheorem{rem}[thm]{Remark}

\newcommand{\Tr}{\operatorname{Tr\,}}

\begin{document}
\title[Isoperimetric inequalities for Riesz potentials]
{Isoperimetric inequalities for Schatten norms of Riesz potentials}

\author[G. Rozenblum]{G. Rozenblum}
\address{
  Grigori Rozenblum:
  \endgraf
  Department of Mathematics
  \endgraf
  Chalmers University of Technology and The University of Gothenburg
  \endgraf
  Chalmers Tvargatan, 3, S-412 96, Gothenburg
  \endgraf
  Sweden
  \endgraf
  {\it E-mail address} {\rm grigori@chalmers.se}
  }
\author[M. Ruzhansky]{M. Ruzhansky}
\address{
  Michael Ruzhansky:
  \endgraf
  Department of Mathematics
  \endgraf
  Imperial College London
  \endgraf
  180 Queen's Gate, London SW7 2AZ
  \endgraf
  United Kingdom
  \endgraf
  {\it E-mail address} {\rm m.ruzhansky@imperial.ac.uk}
  }
\author[D. Suragan]{D. Suragan}
\address{
  Durvudkhan Suragan:
 \endgraf
  Institute of Mathematics and Mathematical Modelling
  \endgraf
  125 Pushkin str., 050010 Almaty, Kazakhstan
  \endgraf
  and
  \endgraf
  Department of Mathematics
  \endgraf
  Imperial College London
  \endgraf
  180 Queen's Gate, London SW7 2AZ
  \endgraf
  United Kingdom
  \endgraf
  {\it E-mail address} {\rm d.suragan@imperial.ac.uk}}

\thanks{The second and the third authors were supported in parts by the EPSRC
 grant EP/K039407/1 and by the Leverhulme Grant RPG-2014-02,
 as well as by the MESRK grant 5127/GF4.}


     \begin{abstract}
     In this note we prove that the ball is a maximiser of some Schatten $p$-norms of the Riesz potential operators among all domains of a given measure in $\mathbb R^{d}$. In particular, the result is valid for the polyharmonic Newton potential operator, which is related to a nonlocal boundary value problem for the poly-Laplacian extending the one considered by M. Kac in the case of the Laplacian, so we obtain  and isoperimetric inequalities for its eigenvalues as well, namely, analogues of Rayleigh-Faber-Krahn and Hong-Krahn-Szeg{\"o} inequalities.
     \end{abstract}
     \maketitle

\section{Introduction}

\subsection{} Let $\Omega\subset\mathbb{R}^{d}$ be a set with finite Lebesgue measure. In $L^2(\Omega)$ we consider the Riesz potential operators
\begin{equation}\label{pot}
(\mathcal{R}_{\alpha,\Omega}f)(x):=\int_{\Omega}\varepsilon_{\alpha,d}(|x-y|)f(y)dy,\quad f\in L^{2}(\Omega),\quad 0<\alpha<d,
\end{equation}
where
\begin{equation}\label{Kernel}
    \varepsilon_{\alpha,d}(|x-y|)=c_{\alpha,d}|x-y|^{\alpha-d}
\end{equation}
and $c_{\alpha,d}$ is a positive constant,
$$c_{\alpha,d}=2^{\alpha-d}\pi^{{-d/2}}\frac{\Gamma(\alpha/2)}{\Gamma((d-\alpha)/2)}.$$
  The operator $\mathcal{R}_{\alpha,\Omega}$ generalises  the Riemann-Liouville ones to several variables and  the Newton potential operators to fractional orders. Since $ c_{\alpha,d}{|x|^{\alpha-d}}$ is the Fourier transform of the function $|\xi|^{-\alpha}$ in $\mathbb{R}^d$, it is the fundamental solution of  $(-\Delta)^{\alpha/2}$, i.e. $(-\Delta_y)^{\alpha/2}\varepsilon_{\alpha,d}(x-y) =\delta_x$. In particular, for an even integer  $\alpha=2m$ with  $ 0<m<d/2$, the function
\begin{equation}
\varepsilon_{2m,d}(|x|)=c_{2m,d} |x|^{2m-d},
\label{i2}
\end{equation}
 is the fundamental solution to the
polyharmonic equation of order $2m$ in $\mathbb R^d$.

So, the polyharmonic Newton potential
\begin{equation}
(\mathcal{L}^{-1}_{2m,\Omega} f)(x):=\int_{\Omega}\varepsilon_{2m,d}(|x-y|)f(y)dy,\, f\in L^{2}(\Omega),
\label{i3}
\end{equation}
is a particular case of the Riesz potential,
\begin{equation}\label{EQ:RL}
\mathcal{L}^{-1}_{2m,\Omega} =\mathcal{R}_{2m,\Omega}.
\end{equation}

In the case $m=1$, i.e., for the Laplacian, under the assumption of a sufficient regularity of the boundary of $\Omega$ (for example, piecewise $C^1$), it is known, see for example Mark Kac  \cite{Kac1},
that the equation
\begin{equation}\label{i4}
u(x)=(\mathcal{L}^{-1}_{2,\Omega} f)(x)=\int_{\Omega}\varepsilon_{2,d}(|x-y|)f(y)dy
\end{equation}
is equivalent to the equation
\begin{equation}\label{i5}
-\Delta u(x)=f(x), \,\,\,\ x\in\Omega,
\end{equation}
with the following  nonlocal integral boundary condition
\begin{equation}
-\frac{1}{2}u(x)+\int_{\partial\Omega}\frac{\partial\varepsilon_{2,d}(|x-y|)}{\partial n_{y}}u(y)d S_{y}-
\int_{\partial\Omega}\varepsilon_{2,d}(|x-y|)\frac{\partial u(y)}{\partial n_{y}}d S_{y}=0,\,\,
x\in\partial\Omega,
\label{i6}
\end{equation}
where $\frac{\partial}{\partial n_{y}}$ denotes the outer normal
derivative at the point $y\in\partial\Omega$.
This approach was further expanded in Kac's book \cite{Kac2}
with   interesting applications, in particular, to the Weyl spectral asymptotics of the  eigenvalue
counting function for the Laplacian, see more in \cite{KS1} and \cite{Sa}.

Considerations in the present note, as it concerns integer values  $m\geq 1$, take care of a generalization of the boundary value problem \eqref{i5}--\eqref{i6}. Moreover, for  noninteger values of  $m$ (i.e., $\alpha\not\in 2\mathbb{Z}$), the operator \eqref{pot} acts as the interior term in the resolvent for  boundary problems for the fractional power of the Laplacian, see, e.g. \cite{Schulze}.

 \subsection{}We are interested in questions of spectral geometry. The main reason why the results are useful, beyond the intrinsic interest of geometric extremum problems, is that they produce \emph{a priori} bounds for spectral invariants of operators on arbitrary domains. For a good general review of isoperimetric inequalities for the Dirichlet, Neumann and other Laplacians we can refer to  \cite{Ben}.

\medskip
This note is motivated in part by a recent paper of Miyanishi and Suzuki \cite{MS}, where
it has been proved that the Schatten norm of the double layer potential is minimised on a
circle among all domains of a given area in $\mathbb R^2$.
In the case of $\Omega\subset\mathbb R^2$, a similar extremum problem  for the logarithmic potential
has been investigated by the second- and third named authors in \cite{Ruzhansky-Suragan-log}.

\medskip
The main results of this note consist in  showing that (under certain restrictions for indices)
the Schatten norms of the Riesz potentials $\mathcal{R}_{\alpha,\Omega}$
over sets of a given measure are maximised on balls.
More precisely, we can summarise our results as follows:

\begin{itemize}

\item
Let $0<\alpha<d$ and let $\Omega^{*}$ be a ball in $\mathbb R^d$;
we set $p_0:=d/\alpha$.
Then for any integer $p$ with $p_0< p\leq\infty$ we have
\begin{equation}
\|\mathcal{R}_{\alpha,\Omega}\|_{p}\leq  \|\mathcal{R}_{\alpha,\Omega^{*}}
\|_{p},
\label{eqthm0}
\end{equation}
for any  domain $\Omega$ with $|\Omega|=|\Omega^{*}|.$
Here $\|\cdot\|_p$ stands for the Schatten $p$-norm, $|\cdot|$ for the Lebesgue measure. The proof is based on  the application of a suitably adapted Brascamp-Lieb-Luttinger inequality. Note that for $p=\infty$ this result gives a variant of the famous Rayleigh-Faber-Krahn inequality for the
 Riesz potentials (and hence also the Newton potential).

\item We also establish the Hong-Krahn-Szeg\"{o} inequality: the
maximum of the second eigenvalue of $\mathcal{R}_{\alpha,\Omega}$ among bounded open sets with a given measure is approached by the union of two identical balls with mutual distance going to infinity.
\end{itemize}
\subsection{}
There is a vast amount of papers dedicated to above type of results for Dirichlet, Neumann and other Laplacians,
see, for example, \cite{BF}, \cite{He} and references therein.
For instance, the questions are still open for boundary value problems of the bi-Laplacian (see \cite[Chapter 11]{He}). The main difficulty arises because the resulting operators of these boundary value problems are not positive for higher powers of the Laplacian. The same is the situation  for the Schatten $p$-norm inequalities: the result for the Dirichlet Laplacian can be obtained from Luttinger's inequality \cite{Lu} but very little is known for other Laplacians (see \cite{Dit1}). The Hong-Krahn-Szeg\"{o} inequality for the Robin Laplacian was proved recently \cite{Ken09} (see \cite{BP} for further discussions).
So, in general, until now there were no examples of a boundary value problem for the poly-Laplacian ($m>1$) for which all the above results had been proved. It seems that there are no isoperimetric results for the fractional order Riesz potentials either.

We believe that  Kac's boundary value problem \eqref{n1}-\eqref{n3} serves as the first example of such boundary value problem, for which all
the above results are true.
This problem
describes the nonlocal boundary conditions for the poly-Laplacian corresponding to the polyharmonic Newton potential operator.

\smallskip
In a bounded connected domain $\Omega\subset\mathbb R^{d}$ with a piecewise $C^{1}$ boundary $\partial\Omega$, as an analogue to \eqref{i5}
we consider the polyharmonic equation
\begin{equation}
(-\Delta_{x})^{m}u(x)=f(x), \quad x\in\Omega,\quad m\in\mathbb{N}.
\label{n1}
\end{equation}
To relate the polyharmonic Newton potential \eqref{i3} to the boundary value problem
\eqref{n1} in $\Omega$, we can use the result of \cite{ KS2} asserting that
for each function  $f\in L^2(\Omega)$, the polyharmonic Newton potential \eqref{i3}
belongs to the class  $H^{2m}(\Omega)$  and satisfies,
for $i=0,1,\ldots,m-1,$ the nonlocal boundary conditions
\begin{multline}\label{n3}
-\frac{1}{2}(-\Delta_{x})^{i}u(x)+
\\
+\sum_{j=0}^{m-i-1}
\int_{\partial\Omega}\frac{\partial}{\partial
n_{y}}(-\Delta_{y})^{m-i-1-j}\varepsilon_{2(m-i),d}(|x-y|)
(-\Delta_{y})^{j}(-\Delta_{y})^{i}u(y)dS_{y}
\\
-\sum_{j=0}^{m-i-1}\int_{\partial\Omega}(-\Delta_{y})^{m-i-1-j}
\varepsilon_{2(m-i),d}(|x-y|)\frac{\partial}{\partial
n_{y}}(-\Delta_{y})^{j}(-\Delta_{y})^{i}u(y)dS_{y}=0,\,\,
 x\in \partial\Omega.
\end{multline}
Conversely, if a function  $u\in H^{2m}(\Omega)$ satisfies \eqref{n1} and the boundary conditions \eqref{n3}
for $i=0,1,\ldots,m-1,$ then it defines the polyharmonic Newton potential by the formulae \eqref{i3}.

\medskip
Therefore, our analysis (of the special case of the Riesz potential) of the polyharmonic Newton potential \eqref{i3}
implies corresponding result for the boundary value problem \eqref{n1}--\eqref{n3}. Note that the analogue of the problem \eqref{n1}--\eqref{n3} for the Kohn Laplacian
and its powers on the Heisenberg group have been recently investigated in
\cite{Ruzhansky-Suragan:Kohn-Laplacian}.
We note that there are  certain interesting questions concerning such operators,  lying beyond  Schatten classes  properties, see e.g. \cite{Dostanic-AM} for different regularised trace formulae.

\medskip
\subsection{} In Section \ref{SEC:result} we discuss main spectral properties of  the Riesz potential operator and formulate the main results of this note.
Their proof will be given in Sections \ref{SEC:proof} and \ref{SEC:5}.

\medskip
The authors would like to thank Rupert Frank for comments.

\section{Main results}\label{SEC:result}
\subsection{The operator $\mathcal{R}_{\alpha,\Omega}$ and its properties}

 We study the spectral problem of the Riesz potentials
\begin{equation}\label{n5}
 \mathcal{R}_{\alpha,\Omega} u=\int_{\Omega}\varepsilon_{\alpha}(|x-y|)u(y)dy=\lambda u, \,u\in L^2(\Omega),
\end{equation}
where $$\varepsilon_{\alpha,d}(|x-y|):={c_{\alpha,d}}{|x-y|^{\alpha-d}}$$ and $0<\alpha<d$.
(we may sometimes drop the subscripts $\alpha,d$ and $\Omega$ in the notation of the operator and the kernel, provided this does not cause a confusion.)
Recall that in the case of the Newton potential operator it is the same as considering the spectrum of the operator corresponding to the boundary value problem \eqref{n1}--\eqref{n3}, which we call $\mathcal{L}=\mathcal{L}_{2m,\Omega}$,
 in a bounded connected domain $\Omega\subset
\mathbb R^{d}$ with a piecewise $C^{1}$ continuous boundary $\partial\Omega$, that is,
\begin{equation}\label{n4}
(-\Delta_{x})^{m}u(x)=\lambda^{-1} u(x),\,\, x \in\Omega,\,\,m\in\mathbb{N},
\end{equation}
with the nonlocal boundary conditions \eqref{n3}.

\smallskip
Recall that $\Omega$ has finite Lebesgue measure. The well-known Schur test shows immediately  that $\mathcal{R}_{\alpha,\Omega}$ is bounded in $L^2(\Omega).$ Moreover, it will be shown soon that this operator is compact in $L^2(\Omega)$ as well and belongs to certain Schatten classes $\mathfrak{S}^p$. Since the Riesz kernel is symmetric, the operator $\mathcal{R}_{\alpha,\Omega}$ is self-adjoint.

Recall that the norm in Schatten class $\mathfrak{S}^p$ (the $p$-norm) of a compact operator $T$  is defined as
\begin{equation}
\| T\|_{p}=\left(\sum_{j=1}^{\infty}s_{j}(T)^{p}\right)^{\frac{1}{p}}, \quad 1\leq p <\infty,\label{5}
\end{equation}
for $s_{1}\geq s_{2}\geq...>0$ being the the singular values of $T$. For $p=\infty$, it is usual to set
$$
\| T\|_{\infty}:=\|T\|,
$$
i.e.,   the operator norm of $T$ in $L^2(\Omega)$.
For compact  self-adjoint operators the singular values are equal to the moduli of (nonzero) eigenvalues, and the corresponding eigenfunctions form a complete orthogonal basis on $L^2$. Additionally, if the operator is non-negative, the words `moduli of' in the previous sentence may be deleted.

\medskip

 The eigenvalues  of $\mathcal{R}_{\alpha,\Omega}$ may be enumerated in the descending order of their moduli,
$$|\lambda_{1}|\geq|\lambda_{2}|\geq...$$
where $\lambda_{j}$ is repeated in this series according to its multiplicity. We denote the
corresponding eigenfunctions by $u_{1}, u_{2},\ldots,$ so that for each eigenvalue $\lambda_{j}$  one and only one
corresponding (normalised) eigenfunction $u_{j}$ is fixed,
$$\mathcal{R}_{\alpha,\Omega}u_j=\lambda_j u_j,\,\,\,\, j=1,2,....$$
The following proposition asserts that the operator $\mathcal{R}_{\alpha,\Omega}$ is compact and evaluates  the decay  rate of its singular numbers.
\begin{prop}\label{Prop.Schatten}Let $\Omega\subset\mathbb{R}^d$ be a measurable set with finite Lebesgue measure, $0<\alpha<d$. Then
\begin{enumerate}
\item The operator $\mathcal{R}_{\alpha,\Omega}$ is nonnegative; this means, in particular, that all eigenvalues are nonnegative, $$\lambda_j\equiv\lambda_j(\mathcal{R}_{\alpha,\Omega})=|\lambda_j(\mathcal{R}_{\alpha,\Omega})|=s_j.$$
    \item For the eigenvalues $\lambda_j$ the following estimate holds:
$$\lambda_j\le C |\Omega|^{\vartheta}j^{-\vartheta},$$
where $\vartheta=\alpha/d$. (In particular, this implies the compactness of the operator.)
\end{enumerate}
\end{prop}
\begin{proof}We start by recalling that
\begin{equation}\label{convolution}
    \varepsilon_{\alpha',d}*\varepsilon_{\alpha'',d}(|x-y|)\equiv\int_{\mathbb{R}^d}\varepsilon_{\alpha',d}(|x-z|)\varepsilon_{\alpha'',d}(|z-y|)dz=(2\pi)^d\varepsilon_{\alpha'+\alpha'',d}(|x-y|),
\end{equation}
as soon as $0<\alpha',\alpha''<\alpha'+\alpha''<d.$  Since $|\xi|^{-(\alpha'+\alpha'')}=|\xi|^{-\alpha'}|\xi|^{-\alpha''}$, this well known relation follows, for example, from the fact, already mentioned,  that $\varepsilon_{\alpha,d}$ is the Fourier transform of $|\xi|^{-\alpha}$, and from the relation between the Fourier transform of a product  and the convolution of the Fourier transforms.
We denote by $\chi_\Omega$ the characteristic function of the set $\Omega$. Consider the operator
\begin{equation}\label{tilde}
\tilde{\mathcal{R}}_{\alpha,\Omega}:\quad f\in L^2(\mathbb{R}^d)\mapsto \chi_\Omega(x)\int_{\mathbb{R}^d}\varepsilon_{\alpha,d}(|x-y|)\chi_\Omega(x)(y)f(y)dy \in L^2(\mathbb{R}^d).
\end{equation}
In the direct sum decomposition $L^2(\mathbb{R}^d)=L^2(\Omega)\oplus L^2(\mathbb{R}^d\setminus\Omega)$ the operator $\tilde{\mathcal{R}}_{\alpha,\Omega}$ is represented as $\mathcal{R}_{\alpha,\Omega}\oplus \pmb{0},$
therefore the (nonzero) singular numbers of operators $\tilde{\mathcal{R}}_{\alpha,\Omega}$ and $\mathcal{R}_{\alpha,\Omega}$ coincide.  Due to \eqref{convolution}, the operator $\tilde{\mathcal{R}}_{\alpha,\Omega}$ can be represented as $ (2\pi)^dT^*T$, where $T:L^2(\mathbb{R}^d)\to L^2(\mathbb{R}^d),\quad$
\begin{equation}\label{T*T}
    Tf(x)=\int\varepsilon_{\alpha/2,d}(|x-y|)\chi_\Omega(y)f(y)dy  .
\end{equation}
The above relations show that, in fact, the operator $\tilde{\mathcal{R}}_{\alpha,\Omega}=T^*T$ and, further on, the operator ${\mathcal{R}}_{\alpha,\Omega}$ are nonnegative; this proves the first statement in the Proposition.
Further on, the eigenvalues  of $\mathcal{R}_{\alpha,\Omega}$ equal the squares of the singular numbers of $T$. Now we can apply the Cwikel estimate, see \cite{Cwikel}, concerning the singular numbers estimates for integral operators with kernel of the form $h(x-y)g(y)$. In our case, $h=\varepsilon_{\alpha/2,d},\quad$ $g=\chi_\Omega$, and thus the Cwikel estimate (1) in \cite{Cwikel}, with $p=2d/\alpha$ gives
\begin{equation}\label{Cw.T}
    s_j(T)\le Cj^{-1/p} \|\chi_\Omega\|_{L^p} =Cj^{-\alpha/(2d)}|\Omega|^{\alpha/(2d)},
\end{equation}
with certain constant $C=C(\alpha,d)$ and, finally,
\begin{equation}\label{est}
    s_j(\mathcal{R}_{\alpha,\Omega})\le C j^{-\theta}|\Omega|^{\theta}, \quad \theta = \alpha/d.
\end{equation}
The proof is complete.
\end{proof}
\begin{rem}Actually, for the operator $T$, and therefore, for the operator $\mathcal{R}_{\alpha,\Omega}$, the estimate  \eqref{est} is accompanied by the asymptotic formula, $s_j(T)\sim C j^{-\alpha/(2d)}|\Omega|^{\alpha/(2d)}$, with an explicitly given constant $C$. For a bounded set $\Omega$, this asymptotics is a particular case of  general results of M.Birman-M.Solomyak's paper \cite{BirSol} concerning integral operators with weak polarity in the kernel. One can easily dispose of this boundedness condition, using the estimate \eqref{est} and the asymptotic approximation procedure, like this was done many times since early 70-s, for example, in \cite{BirSolDif} and \cite{RSV4}.
\end{rem}

It follows from Proposition \ref{Prop.Schatten} that the operator $\mathcal{R}_{\alpha,\Omega}$ belongs to every Schatten class $\mathfrak{S}^p$ with $p>p_0=\alpha/d$ and
\begin{equation}
\|\mathcal{R}_{\alpha,\Omega}\|_{p}=\left(\sum_{j=1}^{\infty}{\lambda_{j}(\Omega)^{p}}\right)^{\frac{1}{p}},\,\, 1\leq p <\infty.\label{6}
\end{equation}

In the course of our further analysis, we will  need to calculate the trace of certain trace class integral operators. For a positive trace class integral operator $\mathbf{K}$ with continuous kernel $K(x,y)$ on a nice set $\Omega$, it is well known that $\Tr(\mathbf{K})=\int_{\Omega} K(x,x) dx.$ This result cannot be used in our more general setting, since our set $\Omega$ is not supposed to be nice and we cannot grant the continuity of the kernels in question. So we need some additional work.

For an operator $\mathbf{K}$ with kernel $K(x,y)$ an exact criterium for membership in $\mathfrak{S}^p$ in  terms of the kernel exists only for $p=2$, i.e., for Hilbert-Schmidt operators (but see also conditions for Schatten classes in terms of the regularity of the kernel in \cite{DR} and in Remark \ref{REM:Kt}). Namely for $\mathbf{K}$ to belong to $\mathfrak{S}^2$, it is necessary and sufficient that $\int\!\!\int_{\Omega\times\Omega}|K(x,y)|^2dxdy<\infty$, moreover, $\|\mathbf{K}\|_2^2$ equals exactly the above integral and for the trace class operator $\mathbf{K}^*\mathbf{K}$ the same  integral equals its trace. We are going to discuss some variations on these well known properties.

First, let us recall a sufficient condition for an integral operator to belong to the Schatten class $\mathfrak{S}^p$ for $p>2$. This condition was found independently by several mathematicians; for us, it is convenient to refer to the paper \cite{Russo}.  The class $L^{p,q}$ is defined as consisting of functions $K(x,y),$ $x,y\in\Omega$ such that $\|K\|_{L^{p,q}}=\left(\left(|K(x,y)|^pdx\right)^{q/p}\right)^{1/q}<\infty$. The main result in \cite{Russo} asserts the following.
\begin{thm}\label{RussoTh} Let $p>2$, $p'=p/(p-1)$ and let the kernel $K$ belong to $L^2(\Omega\times\Omega)$. Suppose that $K$ and the adjoint kernel $K^*(x,y)=\overline{K(y,x)}$ belong to $L^{p',p}$. Then the integral operator $\mathbf{K}$ with kernel $K(x,y)$ belongs to the Schatten class $\mathfrak{S}^p$. Moreover, $\|\mathbf{K}\|_p\le(\|K\|_{L^{p',p}}\|K^*\|_{L^{p',p}})^{\frac12}$
\end{thm}
What we, actually, need are some consequences of this Theorem, established in \cite{Goffeng}. First of all, it is shown there, quite elementarily,  that the condition $K\in L^2(\Omega\times\Omega)$ is excessive and may be removed. However, what we need most is the following result (see Theorem 2.4 in \cite{Goffeng} for the statement in most generality).

\begin{thm}\label{goffength}
Let the kernel $K(x,y)$ satisfy the conditions of Theorem \ref{RussoTh} for some $p>2$. Then for the operator $\mathbf{K}^s$, which belongs to the  trace class by this theorem  for any integer $s>p$, the following formula holds
\begin{equation}\label{TraceGoffeng}
    \Tr(\mathbf{K}^s)=\int\limits_{\Omega^s}\left(\prod_{k=1}^{s}K(x_k,x_{k+1})\right)dx_1dx_2\dots dx_s, \, x_{s+1}\equiv x_1.
\end{equation}
\end{thm}

We apply Theorem \ref{goffength} to our kernel $K(x,y)=\varepsilon_{\alpha,d}(|x-y|)$, $x,y\in\Omega$. Since the measure of $\Omega$ is finite, the kernel $K(x,y)$ belongs to $L^{p',p}(\Omega\times\Omega)$ for any $p>\frac{d}{\alpha}$. Therefore, for the trace of  the operator $\mathbf{K}^s$  formula \eqref{TraceGoffeng} holds, and thus, for $s>p_0=\frac{d}{\alpha}$ we have
\begin{multline}\label{traceintegral}
   \sum{\lambda_j(\mathcal{R}_{\alpha,\Omega})^s} = \Tr(\mathcal{R}_{\alpha,\Omega}^s) \\
   =\int\limits_{\Omega^s} \left(\prod_{k=1}^{s}K(x_k,x_{k+1})\right)dx_1dx_2\dots dx_s, \, x_{s+1}\equiv x_1.
\end{multline}

\begin{rem}\label{REM:Kt}
For the membership in the Schatten classes $\mathfrak{S}^p$ with $p<2$ usually a certain regularity of the
kernel is required. In \cite{DR} it was shown, among other things, that if the integral kernel $K$ of an operator
$\mathbf{K}f(x)=\int_\Omega K(x,y) f(y) dy$ satisfies $$K\in H^\mu(\Omega\times\Omega)$$ for a manifold $\Omega$ of dimension $d$ then
$$\mathbf{K}\in \mathfrak{S}^p(L^2(\Omega)) \textrm{  for } p>\frac{2d}{d+2\mu}.$$
In the case of the Riesz potential with $K(x,y)=\varepsilon_{\alpha,d}(|x-y|)$ it can be readily checked that
it implies that
$\mathcal{R}_{\alpha,\Omega}\in \mathfrak{S}^p(L^2(\Omega))$ for $p>\frac{d}{\alpha}$.

As it was already mentioned before if the integral kernel $K^{(s)}$ of the operator $\mathbf{K}^s$ is not continuous,
the formula $\Tr(\mathbf{K}^s)=\int_{\Omega} K^{(s)}(x,x) dx$ may not hold but it can be replaced by
the formula \eqref{TraceGoffeng} (and hence also \eqref{traceintegral}).
However, we can mention another expression for the trace: if $\widetilde{K^{(s)}}$ denotes the averaging
of $K^{(s)}$ with respect to the martingale maximal function, we have
$$
\Tr(\mathbf{K}^s)=\int_{\Omega} \widetilde{K^{(s)}}(x,x) dx,
$$
where we refer to \cite[Section 4]{DR} for the description of $\widetilde{K^{(s)}}$, its properties, and
further references.
\end{rem}

\subsection{Formulation of  main results} We now formulate the main results of this note. Here $|\Omega|$ denotes
the Lebesgue measure of $\Omega$.

\begin{thm} \label{THM:main} Let $\Omega^{*}$ be a ball in $\mathbb R^d$.
Let $p_0:=\frac{d}{\alpha}$.
Then for any integer $p$ with $p_0< p\leq\infty$, we have
\begin{equation}
\|\mathcal{R}_{\alpha,\Omega}\|_{p}\leq  \|\mathcal{R}_{\alpha,\Omega^{*}}\|_{p},\label{eqthm}
\end{equation}
for any domain $\Omega$ with $|\Omega|=|\Omega^{*}|.$
\end{thm}

For $p=\infty$, the statement will follow from a variant of the Rayleigh-Faber-Krahn inequality
for the operator $\mathcal{R}_{\alpha,\Omega}$.
For all finite  integers $p$, Theorem \ref{THM:main} follows from the Brascamp-Lieb-Luttinger inequality for symmetric rearrangements of $\Omega$.

\smallskip
We also obtain the following Hong-Krahn-Szeg\"{o} inequality:

\begin{thm}\label{THM:second}
The maximum of the second eigenvalue $\lambda_{2}(\Omega)$ of
$\mathcal{R}_{\alpha,\Omega}$ among  all sets  $\Omega\subset\mathbb R^{d}$ with a given measure
is approached by the union of two identical balls with mutual distance going to infinity.
\end{thm}

Similar result for the Dirichlet Laplacian is called the Hong-Krahn-Szeg\"{o} inequality. See, for example, \cite{BF} and \cite{Ken09} for further references. We also refer to \cite{BP} which deals with the second eigenvalue of a nonlocal (and nonlinear) $p$-Laplacian operator.
We note that in Theorem \ref{THM:second} we have $\lambda_{2}(\Omega)>0$ since all the eigenvalues of $\mathcal{R}_{\alpha,\Omega}$ are nonnegative (see Proposition \ref{Prop.Schatten}).

\section{Proof of Theorem \ref{THM:main}}
\label{SEC:proof}
Since the integral kernel of $\mathcal{R}_{\alpha,\Omega}$ is positive, the statement, sometimes called Jentsch's theorem, applies, see, e.g.,  \cite{RSV4}.

\begin{lem} \label{lem:1} The eigenvalue $\lambda_{1}$ of $\mathcal{R}_{\alpha,\Omega}$ with the largest modulus
is positive and simple; the corresponding eigenfunction $u_{1}$ is positive, and any other eigenfunction $u_{j},\,\, j> 1,$ is sign changing in $\Omega$.
\end{lem}
(Note that the positivity of $\lambda_1$ is already known, since the operator $\mathcal{R}_{\alpha,\Omega}$ is nonnegative; what is important is the positivity of $u_1$.)

We will also use this Lemma in Section \ref{SEC:5}. Recall that we have already established  in Proposition \ref{Prop.Schatten} that   all $\lambda_{j}(\Omega),\,i=1,2,...,$ are positive for any domain $\Omega$.

\medskip
Now we prove the following analogue of  Rayleigh-Faber-Krahn theorem for the operator $\mathcal{R}_{\alpha,\Omega}$. We will use this fact further on,  in Section \ref{SEC:5}.
See \cite{Ben} for a general discussion of this subject.

\begin{lem} \label{lem:2} The ball $\Omega^{*}$ is a
maximiser of the first eigenvalue of the operator $\mathcal{R}_{\alpha,\Omega}$ among all domains of a given volume, i.e.
$$
0<\lambda_{1}(\Omega{})\leq \lambda_{1}(\Omega^*)
$$
for an arbitrary domain $\Omega\subset \mathbb R^{d}$ with $|\Omega|=|\Omega^{*}|.$
\end{lem}

\begin{rem}\label{REM:op}
In other words Lemma \ref{lem:2} says that the operator norm of
$\mathcal{R}_{\alpha,\Omega}$ is maximised in the ball among all Euclidean domains of a given volume.
\end{rem}

\begin{proof}[Proof of Lemma \ref{lem:2}] Let $\Omega$ be a bounded measurable set in $\mathbb R^{d}.$
Its symmetric rearrangement $\Omega^{\ast}$ is an open ball centred at $0$
with the measure equal to the measure of $\Omega,$ i.e. $|\Omega^{\ast}|=|\Omega|$.
Let $u$ be a nonnegative measurable function in $\Omega$,
such that all its positive level sets have finite measure.
With the definition of the symmetric-decreasing rearrangement of $u$
we can use the layer-cake decomposition \cite{LL},
which expresses a nonnegative function $u$ in terms of its level sets as
\begin{equation}
u(x)=\intop_{0}^{\infty}\chi_{\left\{ u(x)>t\right\}
}dt,
\end{equation}
where $\chi$ is the characteristic function of the corresponding domain.
The function
\begin{equation}
u^{\ast}(x):=\intop_{0}^{\infty}\chi_{\left\{u(x)>t\right\}^{\ast}\label{2}
}dt
\end{equation}
is called the (radially) symmetric-decreasing rearrangement of a nonnegative measurable function $u$.

Recalling the Riesz inequality \cite{LL} and the fact that
$\varepsilon_{\alpha}(|x-y|)$ is a symmetric-decreasing function,  we obtain
\begin{equation}\label{n10}
\int_{\Omega}\int_{\Omega}
u_{1}(y)\varepsilon_{\alpha}(|y-x|)u_{1}(x)dydx
\leq
\int_{\Omega^{\ast}}\int_{\Omega^{\ast}}
u_{1}^{\ast}(y)\varepsilon_{\alpha}(|y-x|)u_{1}^{\ast}(x)dydx.
\end{equation}
In addition, for each nonnegative function $u\in L^2(\Omega)$ we have
\begin{equation}\label{n11}
\| u\|_{L^2(\Omega)}=\| u^{\ast}\|_{L^2(\Omega^{\ast})}.
\end{equation}
Therefore, from \eqref{n10}, \eqref{n11} and the variational principle for $\lambda_{1}(\Omega^{\ast})$, we get
$$
\lambda_{1}(\Omega) =\frac{\int_{\Omega}\int_{\Omega}
u_{1}(y)\varepsilon_{\alpha}(|y-x|)u_{1}(x)dydx}
{\int_{\Omega}|u_{1}(x)|^{2}dx}\leq$$
$$\frac{\int_{\Omega^{\ast}}\int_{\Omega^{\ast}}
u^{\ast}_{1}(y)\varepsilon_{\alpha}(|y-x|)
u^{\ast}_{1}(x)dydx}{\int_{\Omega^{\ast}}|u^{\ast}_{1}(x)|^{2}dx}\leq
$$
$$
\sup_{v\in
L^2(\Omega^{\ast}),v\neq 0}\frac{\int_{\Omega^{\ast}}\int_{\Omega^{\ast}}
v(y)\varepsilon_{\alpha}(|y-x|)
v(x)dydx}{\int_{\Omega^{\ast}}|v(x)|^{2}dx}={\lambda_{1}(\Omega^{\ast})},
$$
completing the proof.
\end{proof}

Now we can finish the
\emph{Proof} of Theorem \ref{THM:main}. For integer values of $p>p_0$ we have by Theorem \ref{goffength}:

\begin{equation}
\sum_{j=1}^{\infty}{\lambda_{j}^{p}(\Omega)}=\int_{\Omega}...
\int_{\Omega}\varepsilon_{\alpha}(|y_{1}-y_{2}|)...\varepsilon_{\alpha}(|y_{p}-y_{1}|)dy_{1}...dy_{p},\,\,\, p> p_{0}, \,\,p\in
\mathbb N.\label{eq315}
\end{equation}
It follows from the Brascamp-Lieb-Luttinger \cite{BLL} inequality that
$$
\int_{\Omega^p}\varepsilon_{\alpha}(|y_{1}-y_{2}|)...\varepsilon_{\alpha}(|y_{p}-y_{1}|)dy_{1}...dy_{p}\leq$$
\begin{equation}\int_{\Omega^{*}}...
\int_{\Omega*}\varepsilon_{\alpha}(|y_{1}-y_{2}|)...\varepsilon_{\alpha}(|y_{p}-y_{1}|)dy_{1}...dy_{p},\label{eq316}
\end{equation}
which proves
\begin{equation}
\sum_{j=1}^{\infty}{\lambda_{j}^{p}(\Omega)}\leq
\sum_{j=1}^{\infty}{\lambda_{j}^{p}(\Omega^{*})},\,\,\, p\in \mathbb N,\,\,\, p> p_{0},\label{eq317}
\end{equation}
for $\Omega\subset \mathbb R^{d}$ with $|\Omega|=|\Omega^{*}|$.
Here we have used that the kernel $\varepsilon_{\alpha}$ is a symmetric-decreasing function in $\Omega^{*}\times\Omega^{*}$, i.e.
$$\varepsilon_{\alpha}^{*}(|x-y|)=\varepsilon_{\alpha}(|x-y|),\,\,\,\, x, y\in \Omega^{*}\times\Omega^{*}.$$
The proof is complete.

\section{Proof of Theorem \ref{THM:second}}
\label{SEC:5}

To prove Theorem \ref{THM:second} we will use the classical two-ball trick, Lemma \ref{lem:1} and Lemma \ref{lem:2}.

\begin{proof} Let us introduce the following sets:
$$\Omega^{+}:=\{x: u_{2}(x)>0\},\,\,\Omega^{-}:=\{x: u_{2}(x)<0\}.$$
Therefore,
$$u_{2}(x)>0,\,\, \forall x\in\Omega^{+}\subset \Omega, \,\, \Omega^{+}\neq\{\emptyset\},$$
$$u_{2}(x)<0,\,\, \forall x\in\Omega^{-}\subset \Omega, \,\, \Omega^{-}\neq\{\emptyset\},$$
and it follows from Lemma \ref{lem:1} that
the domains $\Omega^{-}$ and $\Omega^{+}$  both have positive Lebesgue measure.
Taking
\begin{equation}
u_{2}^{+}(x):=\left\{
\begin{array}{ll}
    u_{2}(x)\,\,{\rm in}\,\, \Omega^{+},\\
    0 \,\, {\rm otherwise}, \\
\end{array}
\right.
\end{equation}
and
$$
u_{2}^{-}(x):=\left\{
\begin{array}{ll}
    u_{2}(x)\,\,{\rm in}\,\, \Omega^{-},\\
    0 \,\, {\rm otherwise}, \\
\end{array}
\right.
$$
we obtain
$$\lambda_{2}(\Omega)u_{2}(x)=\int_{\Omega^{+}}\varepsilon_{\alpha}(|x-y|)u_{2}^{+}(y)dy+
\int_{\Omega^{-}}\varepsilon_{\alpha}(|x-y|)u_{2}^{-}(y)dy,\,\, x\in\Omega.$$
Multiplying by $u_{2}^{+}(x)$ and integrating over $\Omega^{+}$ we get

\begin{multline*}
\lambda_{2}(\Omega)\int_{\Omega^{+}}|u_{2}^{+}(x)|^{2}dx=
\int_{\Omega^{+}}u_{2}^{+}(x)\int_{\Omega^{+}}\varepsilon_{\alpha}(|x-y|)u_{2}^{+}(y)dydx \\
+\int_{\Omega^{+}}u_{2}^{+}(x)
\int_{\Omega^{-}}\varepsilon_{\alpha}(|x-y|)u_{2}^{-}(y)dydx,\,\, x\in\Omega.
\end{multline*}
The second term on the right hand side is non-positive since  the integrand is non-positive.
Therefore,

$$\lambda_{2}(\Omega)\int_{\Omega^{+}}|u_{2}^{+}(x)|^{2}dx\leq
\int_{\Omega^{+}}u_{2}^{+}(x)\int_{\Omega^{+}}\varepsilon_{\alpha}(|x-y|)u_{2}^{+}(y)dydx,$$
that is,
$$\frac{\int_{\Omega^{+}}u_{2}^{+}(x)\int_{\Omega^{+}}\varepsilon_{\alpha}(|x-y|)u_{2}^{+}(y)dydx}{\int_{\Omega^{+}}|u_{2}^{+}(x)|^{2}dx}\geq\lambda_{2}(\Omega).$$
By the variational principle,
$$\lambda_{1}(\Omega^{+})=
\sup_{v\in L^{2}(\Omega^{+}), v\not\equiv 0}\frac{\int_{\Omega^{+}}v(x)\int_{\Omega^{+}}\varepsilon_{\alpha}(|x-y|)v(y)dydx}{\int_{\Omega^{+}}|v(x)|^{2}dx}$$
$$\geq\frac{\int_{\Omega^{+}}u_{2}^{+}(x)\int_{\Omega^{+}}\varepsilon_{\alpha}(|x-y|)u_{2}^{+}(y)dydx}{\int_{\Omega^{+}}|u_{2}^{+}(x)|^{2}dx}\geq\lambda_{2}(\Omega).$$
Similarly, we get
$$\lambda_{1}(\Omega^{-})\geq\lambda_{2}(\Omega).$$
So we have
\begin{equation}\label{13}
\lambda_{1}(\Omega^{+})\geq \lambda_{2}(\Omega),\,\,
\lambda_{1}(\Omega^{-})\geq \lambda_{2}(\Omega).
\end{equation}
We now introduce $B^{+}$ and $B^{-}$, the balls of the same volume as $\Omega^{+}$ and $\Omega^{-}$, respectively.
Due to Lemma \ref{lem:2}, we have
\begin{equation}\label{14}
\lambda_{1}(B^{+})\geq\lambda_{1}(\Omega^{+}),\,\,\lambda_{1}(B^{-})\geq\lambda_{1}(\Omega^{-}).
\end{equation}
Comparing \eqref{13} and \eqref{14}, we obtain
\begin{equation}\label{15}
\min \{\lambda_{1}(B^{+}),\,\lambda_{1}(B^{-})\}\geq \lambda_{2}(\Omega).
\end{equation}
Now let us consider the set  $B^{+}\cup B^{-},$ with the balls $B^{\pm}$ placed  at distance $l$, i.e.
$$l= {\rm dist}(B^{+},B^{-}),$$
Denote by $u_1^\circledast$ the first normalised eigenfunction of $\mathcal{R}_{\alpha,B^{+}\cup B^{-}}$ and take $u_{+}$ and $u_{-}$ being the first normalised eigenfunctions of each single ball, i.e., of operators $\mathcal{R}_{\alpha,B^{\pm}}.$  We introduce the function $v^\circledast\in L^2(B^{+}\cup B^{-})$, which equals $u_+$ in $B^{+}$ and $\gamma u_-$ in $B^{-}$. Since the functions $u_+,u_-, u^\circledast$ are positive, it is possible to find a real number $\gamma$ so that $v^\circledast$ is orthogonal to $u_1^\circledast$.
 Observe that
\begin{equation}
\int_{B^{+}\cup B^{-}}\int_{B^{+}\cup B^{-}} v^\circledast(x)v^\circledast(y)
\varepsilon_{\alpha}(|x-y|)dxdy
=\sum_{i=1}^{4}\mathcal{I}_{i},
\end{equation}
where
$$\mathcal{I}_{1}:=\int_{B^{+}}\int_{B^{+}}u_{+}(x)u_{+}(y)
\varepsilon_{\alpha}(|x-y|)dxdy,$$
$$\mathcal{I}_{2}:=\int_{B^{+}}\int_{B^{-}}u_{+}(x)u_{-}(y)
\varepsilon_{\alpha}(|x-y|)dxdy,$$
$$\mathcal{I}_{3}:=\gamma\int_{B^{-}}\int_{B^{+}}u_{-}(x)u_{+}(y)
\varepsilon_{\alpha}(|x-y|)dxdy,$$
$$\mathcal{I}_{4}:=\gamma^{2}\int_{B^{-}}\int_{B^{-}}u_{-}(x)u_{-}(y)
\varepsilon_{\alpha}(|x-y|)dxdy.$$

By the variational principle,
$$\lambda_{2}(B^{+}\cup B^{-})= \sup_{v\in L^{2}(B^{+}\bigcup B^{-}),\,v\perp u_{1},\,\parallel v\parallel=1}{\int_{B^{+}\cup B^{-}}\int_{B^{+}\cup B^{-}}v(x)v(y)\varepsilon_{\alpha}(|x-y|)dxdy}.$$
Since by construction $v^\circledast$ is orthogonal to $u_{1}$,
we get
$$\lambda_{2}(B^{+}\cup B^{-})\geq {\int_{B^{+}\cup B^{-}}\int_{B^{+}\cup B^{-}} v^\circledast(x)v^\circledast(y) \varepsilon_{\alpha}(|x-y|)dxdy}
= {\sum_{i=1}^{4}\mathcal{I}_{i}}.$$
On the other hand, since $u_{+}$ and $u_{-}$ are the first normalised eigenfunctions (by Lemma \ref{lem:1} both are positive everywhere) of each single ball $B^{+}$ and $B^{-}$, we have

$$\lambda_1(B^{\pm})=\int_{B^{\pm}}\int_{B^{\pm}}u_{\pm}(x)u_{\pm}(y)\varepsilon_{\alpha}(|x-y|)dxdy$$

Summarising the above facts, we obtain
\begin{multline}
\lambda_{2}(B^{+}\cup B^{-})\geq\\
\frac{\int_{B^{+}}\int_{B^{+}}u_{+}(x)u_{+}(y)\varepsilon_{\alpha}(|x-y|)dxdy+\gamma^{2}\int_{B^{-}}\int_{B^{-}}u_{-}(x)u_{-}(y)\varepsilon_{\alpha}(|x-y|)dxdy+\mathcal{I}_{2}+\mathcal{I}_{3}}{\lambda_{1}(B^{+})^{-1}\int_{B^{+}}
\int_{B^{+}}u_{+}(x)u_{+}(y)\varepsilon_{\alpha}(|x-y|)dxdy+\gamma^{2}\lambda_{1}(B^{-})^{-1}\int_{B^{-}}\int_{B^{-}}u_{-}(x)u_{-}(y)\varepsilon_{\alpha}(|x-y|)dxdy}.
\end{multline}
Since the kernel $\varepsilon_{\alpha}(|x-y|)$ tends to zero as $x\in B^{\pm}, \ y\in B^{\mp}$
and $l\to\infty$,  we observe that
$$\lim_{l\rightarrow\infty}\mathcal{I}_{2}=\lim_{l\rightarrow\infty}\mathcal{I}_{3}=0,$$
thus
\begin{equation}\label{16}
\lim_{l\rightarrow\infty}\lambda_{2}(B^{+}\bigcup B^{-})\geq \max \{\lambda_{1}(B^{+}),\,\lambda_{1}(B^{-})\},
\end{equation}
where $l= {\rm dist}(B^{+},B^{-}).$
The inequalities \eqref{15} and \eqref{16} imply that the optimal set for $\lambda_{2}$ does not exist.
On the other hand, taking $\Omega\equiv B^{+}\bigcup B^{-}$ with $l= {\rm dist}(B^{+},B^{-})\rightarrow\infty$, and $B^{+}$ and $B^{-}$ being identical, from the inequalities \eqref{15} and \eqref{16}
we obtain
\begin{multline*}
\lim_{l\rightarrow\infty}\lambda_{2}(B^{+}\bigcup B^{-})\geq \min \{\lambda_{1}(B^{+}),\,\lambda_{1}(B^{-})\}=\lambda_{1}(B^{+})
\\
=\lambda_{1}(B^{-})\geq\lim_{l\rightarrow\infty}\lambda_{2}(B^{+}\cup B^{-}),
\end{multline*}
and this implies that the minimising sequence for $\lambda_{2}$ is given by a disjoint union of two identical balls with mutual distance going to $\infty$.
\end{proof}

\end{document}